\documentclass[12pt]{article}
\usepackage{fullpage}
\usepackage{amsthm}
\usepackage{amsfonts}
\usepackage{amsmath}

\theoremstyle{plain}
\newtheorem{theorem}{Theorem}
\newtheorem{lemma}[theorem]{Lemma}
\newtheorem{corollary}[theorem]{Corollary}

\theoremstyle{definition}

\newtheorem{conjecture}[theorem]{Conjecture}

\theoremstyle{remark}

\begin{document}
\title{An inertial lower bound for the chromatic number\\ of a graph}
\author{Clive Elphick \\
\small School of Mathematics \\[-0.8ex]
\small University of Birmingham \\[-0.8ex]
\small Birmingham, U.K. \\
\small \tt C.H.Elphick@bham.ac.uk \\
\and Pawel Wocjan \\
\small Department of Computer Science \\[-0.8ex]
\small University of Central Florida \\[-0.8ex]
\small Florida, U.S.A. \\
\small \tt wocjan@cs.ucf.edu
}

\date{March 6, 2017}  

\maketitle

\abstract{Let $\chi(G$) and $\chi_f(G)$ denote the chromatic and fractional chromatic numbers of a graph $G$, and let $(n^+ , n^0 , n^-)$ denote the inertia of $G$. We prove that:

\[
1 + \max\left(\frac{n^+}{n^-} , \frac{n^-}{n^+}\right) \le \chi(G) \mbox{  and conjecture that  } 1 + \max\left(\frac{n^+}{n^-} , \frac{n^-}{n^+}\right) \le \chi_f(G)
\]

We investigate extremal graphs for these bounds and demonstrate that this inertial bound is not a lower bound for the vector chromatic number. We conclude with a discussion of asymmetry between $n^+$ and $n^-$, including some Nordhaus-Gaddum bounds for inertia.}

\section{Introduction}

Let $G$ be  a graph with $n\ge 2$ vertices, chromatic number $\chi(G)$, fractional chromatic number $\chi_f(G)$ and independence number $\alpha(G)$. Let $A$ denote the adjacency matrix of $G$ and let $\mu_1 \ge  \mu_2 \ge ... \ge \mu_n$ denote the eigenvalues of $A$. The inertia of $A$ is the ordered triple $(n^+, n^0, n^-)$ where $n^+, n^0, n^-$ are the numbers (counting multiplicities) of positive, zero and negative eigenvalues of $A$ respectively. Note that $\mathrm{rank}(A) = n^+ + n^-$ and $\mathrm{nullity}(A) = n^0$.

\section{An inertial lower bound}

\begin{theorem}
Let $G$ be any graph with inertia $(n^+, n^0, n^-)$. Then

\[
1 + \max\left(\frac{n^+}{n^-} , \frac{n^-}{n^+}\right) \le \chi(G).
\]
\end{theorem}

To better organize the proof of this theorem we formulate the following simple lemma. Recall that positive semidefinite matrices are Hermitian matrices with non-negative eigenvalues and $Z^*$ denotes the Hermitian transpose for matrices $Z\in\mathbb{C}^{n\times n}$. 

\begin{lemma}\label{semirank}
Let $X,Y\in\mathbb{C}^{n\times n}$ be two positive semidefinite matrices satisfying $X \succeq Y$, that is, their difference $X - Y$ is positive semidefinite. Then
\[
\mathrm{rank}(X) \ge \mathrm{rank}(Y).
\]
\end{lemma}

\begin{proof}
Assume that $\mathrm{rank}(X) < \mathrm{rank}(Y)$. Then, there exists a non-trivial vector $v$ in the range of $Y$ that is orthogonal to the range of $X$. Consequently,
\[
v^*(X - Y)v = -v^* Y v < 0
\]
contradicting that $X - Y$ is positive semidefinite.
\end{proof}

\begin{proof}[Proof of Theorem 1]
Wocjan and Elphick \cite{wocjan13} proved that there exist $\chi$ unitary matrices $U_i$ such that:

\begin{equation}\label{eq:pawel}
\sum_{i=1}^{\chi-1} U_i A U_i^* = -A.
\end{equation}
Let $v_1, ... , v_n$ denote the eigenvectors of unit length corresponding to the eigenvalues $\mu_1 \ge ... \ge \mu_n$. Let $A = B - C$ where 
\[
B = \sum_{i=1}^{n^+}\mu_iv_iv_i^*\quad\mathrm{and}\quad 
C = \sum_{i=n-n^-+1}^n (-\mu_i)v_iv_i^*.
\] 
Note that $B$ and $C$ are positive semidefinite and that $\mathrm{rank}(B) =n^+$ and $\mathrm{rank}(C) = n^-$. Let

\[
P^+ = \sum_{i=1}^{n^+} v_iv_i^* \mbox{  and  } P^- = \sum_{i=n-n^-+1}^n v_i v_i^*
\]
denote the orthogonal projectors onto the subspaces spanned by the eigenvectors corresponding to the positive and negative eigenvalues respectively. Observe that $B = P^+AP^+$ and $C = -P^-AP^-$.

Equation (\ref{eq:pawel}) can be rewritten as:

\[
\sum_{i=1}^{\chi-1} U_iBU_i^* - \sum_{i=1}^{\chi-1} U_iCU_i^* = C - B.
\]
Multiplying both sides by $P^-$ from the left and the right we obtain:

\[
P^-\sum_{i=1}^{\chi-1} U_iBU_i^* P^- - P^-\sum_{i=1}^{\chi-1} U_iCU_i^* P^- = C.
\]
Using that 
\[
P^-\sum_{i=1}^{\chi-1} U_iCU_i^* P^-
\]
is positive semidefinite, it follows that:

\[
P^-\sum_{i=1}^{\chi-1} U_iBU_i^* P^- \succeq C.
\]
Then using that the rank of a sum is less than or equal to the sum of the ranks of the summands, that the rank of a product is less than or equal to the minimum of the ranks of the factors, and Lemma~\ref{semirank}, we have that ($\chi - 1)n^+ \ge n^-.$ Similarly $(\chi - 1)n^- \ge n^+$ by multiplying equation (\ref{eq:pawel}) by $-1$ and multiplying both sides by $P^+$ from the left and right.
\end{proof}

\begin{corollary}\label{npmcorollary}
Let $G$ be any graph with inertia $(n^+, n^0, n^-)$. Then
\[
\max(n^+ , n^-) \le \frac{n(\chi - 1)}{\chi}.
\]
\end{corollary}

\begin{proof}
\[
\chi(G) \ge 1 + \frac{n^+}{n^-} = \frac{n^- + n^+}{n^-} = \frac{n - n^0}{n - n^+ - n^0} \ge \frac{n}{n - n^+}.
\]
Re-arranging and repeating with $\chi \ge 1 + n^-/n^+$ completes the proof.
\end{proof}

Griffith and Luhanga \cite{griffith11} conjectured that for connected, planar graphs $n^- \le 2n/3$. $K_4$ provides a counter-example, but it follows from this corollary that for planar graphs, $n^- \le 3n/4$.

\begin{corollary}

Let $\xi'(G)$ denote the normalized orthogonal rank of a graph. Then

\[
1 + \max\left(\frac{n^+}{n^-} , \frac{n^-}{n^+}\right) \le \xi'(G) \le \chi(G).
\]
\end{corollary}

\begin{proof}
Wocjan and Elphick \cite{wocjan13} proved that $\chi(G)$ can be replaced with $\xi'(G)$ in equality (\ref{eq:pawel}).

\end{proof}

\begin{corollary}
Let $A=(a_{k\ell})$ denote the adjacency matrix of a graph and let $W=(w_{k\ell})$ denote an arbitrary Hermitian matrix such that $a_{k\ell}=0$ implies $w_{k\ell}=0$ for all $k,\ell$.

\[
1 + \max\left(\frac{n^+(W)}{n^-(W)} , \frac{n^-(W)}{n^+(W)}\right) \le  \chi(G).
\]
\end{corollary}

\begin{proof}
Wocjan and Elphick \cite{wocjan13} proved that $A$ can be replaced with $W$ in equality (\ref{eq:pawel}).
\end{proof}

\subsection{Extremal graphs}

We can compare this lower bound against the well known bound due to Hoffman \cite{hoffman70} that $1 + \mu_1/|\mu_n| \le \chi(G)$, and other bounds due to Nikiforov \cite{nikiforov07}, Kolotilina \cite{kolotilina10}, Wocjan and Elphick \cite{wocjan13}, Ando and Lin \cite{ando15} and Elphick and Wocjan \cite{elphick15}.

Theorem~1 is exact for example for bipartite graphs ($n^+ = n^-$), regular complete $q$-partite graphs ($n^+ = 1 , n^- = q - 1$) and $24$-cell $(n^+ = 5, n^- = 10)$, for which some of these other lower bounds are also exact. However Theorem 1 is also exact for example for the following singular and non-singular graphs, for which none of these other bounds is exact:
\begin{itemize}
\item barbell graphs on 2n vertices ($n^+ = 2 , n^- = 2n - 2$)
\item irregular complete $q$-partite graphs ($n^+ = 1, n^- = q - 1$)
\item various Circulant graphs such as Antiprism(9) ($n^+ = 5 , n^- = 10$)
\item Sextic(16,1) ($n^+ = 3 , n^- = 9$).
\end{itemize}
The full version of the Hoffman bound \cite{hoffman70} is that $\mu_1 + \mu_{n-\chi+2} + \ldots +\mu_n \le 0$, which is exact for all the itemised graphs.

\section{A conjectured stronger bound}

\begin{conjecture}\label{conjecture}
Let $G$ be any graph with inertia $(n^+, n^0, n^-)$. Then
\[
1 + \max\left(\frac{n^+}{n^-} , \frac{n^-}{n^+}\right) \le \chi_f(G).
\]
\end{conjecture}

We can prove this conjecture for non-singular graphs as follows.

\begin{proof}
A non-singular graph has $n^0 = 0$. Cvetkovi\'c \emph{et al.}~\cite{cvetkovic79} (page 88) proved that:

\[
\alpha (G) \le n^0 + \min(n^+ , n^-) = \min(n^+ , n^-).
\]
It is well known that $\chi_f(G) \ge n/\alpha(G)$, with equality for vertex-transitive graphs. Therefore

\[
\chi_f(G) \ge \frac{n}{\alpha} \ge \frac{n}{\min(n^+ , n^-)} = \frac{n^+ + n^-}{\min(n^+ , n^-)} = 1 + \max\left(\frac{n^+}{n^-} , \frac{n^-}{n^+}\right).
\]
\end{proof}

Conjecture~\ref{conjecture} is therefore true, for example, for all strongly regular graphs. Using standard notation, if $G$ is a strongly regular graph with spectrum $(k^1 , r^f , s^g)$ then:

\[
\chi_f \ge 1 + \max\left(\frac{n^+}{n^-} , \frac{n^-}{n^+}\right) = 1 + \max\left(\frac{1+f}{g} , \frac{g}{1+f}\right) = \max\left(\frac{n}{g} , \frac{n}{1+f}\right).
\]

\subsection{Extremal graphs}

Conjecture~\ref{conjecture} is exact for example for the Kneser graphs $K_{p,k}$, whose vertices correspond to the $k$-element subset of a set of $p$ elements, and where two vertices are joined if and only if the corresponding sets are disjoint. It is known that $\chi(K_{p,k}) = p - 2k + 2$ and that $\chi_f(K_{p,k}) = p/k$. The inertia of these graphs (see Godsil and Meagher \cite{godsil15} section 2.10) is as follows:

\[
n^+ = {p-1 \choose k} \mbox{   ;   }  n^0 = 0 \mbox{   ;   } n^- = {p-1 \choose k-1}.
\]
Consequently 

\[
1 + \max\left(\frac{n^+}{n^-} , \frac{n^-}{n^+}\right) = 1 + \frac{n^+}{n^-} = 1 + \frac{p - k}{k}  = \frac{p}{k} = \chi_f.
\]
Conjecture~\ref{conjecture} is also exact for cycles. This is obvious for even cycles. For the odd cycle on $2n+1$ vertices:

\[
1 + \max\left(\frac{n^+}{n^-} , \frac{n^-}{n^+}\right) = 1 + \frac{n+1}{n} = 2 + \frac{1}{n} = \chi_f.
\]

\section{Other graph parameters and an example}

There are many graph parameters that lie between the clique number, $\omega(G)$, and the chromatic number. For example 

\[
\omega(G) \le \chi_v(G) \le \theta(\overline{G}) \le \chi_f(G) \le \chi_c(G) \le \lceil \chi_c(G) \rceil = \chi(G),
\]
where $\chi_v(G)$ is the vector chromatic number, $\theta(\overline{G})$ is the Lov\'asz theta function of the complement of $G$ and $\chi_c(G)$ is the circular chromatic number. (These inequalities are sufficiently well known to be included in the Wikipedia entry on graph coloring and in \cite{brown09}.)

Bilu \cite{bilu06} proved that the Hoffman bound is a lower bound for the vector chromatic number
\[
1 + \frac{\mu_1}{|\mu_n|} \le \chi_v(G).
\]
We can demonstrate, however, using the self-complementary pentagon $(C_5)$, that
\[
1 + \max\left(\frac{n^+}{n^-} , \frac{n^-}{n^+}\right) \not\le \chi_v(G).
\]
The spectrum of $C_5$ is $\left(2, (\frac{-1+\sqrt{5}}{2})^2, (\frac{-1-\sqrt{5}}{2})^2\right)$ and Lov\'asz proved that $\theta(C_5) = \sqrt{5}$. Therefore

\begin{multline*}
\sqrt{5} = 1 + \frac{\mu_1}{|\mu_n|} =  \chi_v(C_5) = \theta(\overline{C_5}) = \theta(C_5) = \sqrt{5} < \\ 2.5 = 1 + \frac{n^+}{n^-}  = \chi_f(C_5) < \chi(C_5) = 3.
\end{multline*}
It is instructive to compare the various lower bounds for $\chi(G)$. Because $C_5$ is regular, the bounds due to Hoffman, Nikiforov and Kolotilina are all equal to each other at $\sqrt{5} = 2.236$. The generalisations of these bounds due to Wocjan and Elphick are also all equal to $\sqrt{5}$, as is a bound due to Lima \emph{et al.} \cite{lima11}. The bound due to Ando and Lin equals $2.1$. The bound in this paper equals $2.5$, because it uses the (integral) inertia of a graph instead of its eigenvalues.  However, since $\chi$ is integral, all of these bounds imply $\chi \ge 3$.

\section{On the asymmetry between $n^+$ and $n^-$}
The inertial bound discussed above suggests symmetry between $n^+$ and $n^-$. In this section we explore asymmetry between $n^+$ and $n^-$, beginning with Nordhaus-Gaddum bounds for inertia. 

In 1956, Nordhaus and Gaddum \cite{nordhaus56} proved that:

\[
2\sqrt{n} \le \chi(G) + \chi(\overline{G}) \le n + 1 \mbox{  and  } n \le \chi(G) \cdot \chi(\overline{G}) \le \frac{(n + 1)^2}{4}.
\]
Similar bounds, now known as Nordhaus-Gaddum type inequalities, have been obtained for numerous graph parameters. For example the survey paper by Aouchiche and Hansen \cite{aouchiche12} reviews scores of Nordhaus-Gaddum type inequalities. The survey does not, however, refer to inertia. 

\begin{theorem}

For any graph $G$, 

\begin{align*}
1     &\le n^+(G) + n^+(\overline{G}) \le n + 1 \\  
0     &\le n^0(G) + n^0(\overline{G}) \le n         \\
n-1  &\le n^-(G) + n^-(\overline{G}).
\end{align*}

\end{theorem}

\begin{proof}
The lower bound for $n^+$ is trivial and is exact for complete graphs. The lower bound for $n^0$ is exact for most graphs since "almost all" graphs have a non-integral spectrum. There does not seem to be a straightforward upper bound for $n^-$. For example the Generalised Petersen (15,4) graph on 30 vertices has $n^-(G) + n^-(\overline{G}) = 37$. 
 
The spectrum of $G$ begins:
\begin{multline*}
\mu_1 \ge \ldots \ge \mu_{n^+} > 0 = \mu_{n^+ + 1} = \ldots = \\ 
\mu_{n^+ + n^0} = 0 > \mu_{n^+ + n^0 + 1} \ge \ldots \ge \mu_{n^+ + n^0 + j} > -1
\end{multline*}
and continues
\[
-1 \ge \mu_{n^+ + n^0 + j + 1} \ge \ldots \ge \mu_n.
\]
One of the Courant-Weyl inequalities (see for example section 2.8 in \cite{brouwer10}) states that for Hermitian matrices $A$ and $B$
\[
\mbox{ if } i + j - n \ge 1  \mbox{ then } \mu_i(A) + \mu_j(B) \le \mu_{i+j-n}(A + B).
\]
Therefore
\[
\mu_i(G) + \mu_{n-i+2}(\overline{G}) \le \mu_2(K_n) = - 1, \mbox{  for  } i \ge 2.
\]
Therefore, for $i \ge 2$, if $\mu_i > -1$ then $\mu_{n-i+2}(\overline{G}) \le -1 - \mu_i(G) < 0$. So

\begin{equation}\label{clive}
n^-(\overline{G}) \ge n^+(G) + n^0(G) + j -1.
\end{equation}
The $-1$ term in (\ref{clive}) follows from the constraint that $i \ge 2$. Consequently
\[
n^-(G) + n^-(\overline{G}) \ge n^-(G) + n^+(G) + n^0(G) + j - 1 \ge n - 1.
\]
Similarly, using (\ref{clive}) again
\begin{eqnarray*}
n^+(G) + n^+(\overline{G}) & = & n^+(G) + n - n^-(\overline{G}) - n^0(\overline{G}) \\ 
& \le & n^+(G) +n -(n^+(G) + n^0(G) + j -1) - n^0(\overline{G}) \\
& \le & n + 1.
\end{eqnarray*}
Finally, using the bounds above
\begin{eqnarray*}
n^0(G) + n^0(\overline{G}) & = & 2n - (n^+(G) + n^+(\overline{G})) -(n^-(G) + n^-(\overline{G})) \\
& \le & 2n - 1 - (n - 1) = n.
\end{eqnarray*}
\end{proof}
We can show that the upper bound for $n^+(G) + n^+(\overline{G})$ and the lower bound for $n^-(G) + n^-(\overline{G})$ are exact for strongly regular graphs as follows.

\begin{corollary}
Let $G$ be a strongly regular graph. Then $n^+(G) + n^+(\overline{G}) = n + 1$ and $n^-(G) + n^-(\overline{G}) = n - 1$.
\end{corollary}

\begin{proof}
Using standard notation let $G = SRG(n, k, \lambda, \mu)$ have spectrum $(k^1 , r^f , s^g)$, where $1$, $f$ and $g$ are multiplicities. It is well known that $\overline{G} = SRG(n, n-k-1 , n-2k+\mu-2, n-2k+\lambda)$. Clearly $1 + f(G) + g(G) = 1 + f(\overline{G}) + g(\overline{G}) = n$. It is well known that:

\begin{eqnarray*}
f & = & \frac{1}{2}\left((n - 1) - \frac{2k + (n - 1)(\lambda - \mu)}{\sqrt{(\lambda - \mu)^2 +4(k - \mu)}}\right) \\
g & = & \frac{1}{2}\left((n - 1) + \frac{2k + (n - 1)(\lambda - \mu)}{\sqrt{(\lambda - \mu)^2 +4(k - \mu)}}\right)
\end{eqnarray*}
A page of algebra demonstrates that $f(G) = g(\overline{G})$ and $g(G) = f(\overline{G})$. Therefore

\[
n^+(G) + n^+(\overline{G}) = 2 + f(G) +f(\overline{G}) = 2 + f(G) + g(G) = n + 1.
\]
It is then immediate that $n^-(G) + n^-(\overline{G}) = n - 1$
\end{proof}

It is clear that $n^- = n - 1$ for $K_n$. However $n^+ = n - 1$ only for $K_2$, because $\mu_1 \ge |\mu_n|$. The goal is therefore to find an upper bound for $n^+$ as a function of $n$, which is not an upper bound for $n^-$. Since the Nordhaus-Gaddum upper bound for $n^+$ is exact for SRGs, it is plausible that some SRGs will be extremal.

SRGs, other than the pentagon, have $\mu_2 \ge 1$ so it seems likely that SRGs that maximise $n^+$ have $\mu_2 = 1$. A couple of pages of simple algebra demonstrates that when $\mu_2 = 1$ then

\[
n^+ = n - \left(\frac{n - \lambda + 2\mu}{2 - \lambda + \mu}\right) \mbox{  and  } k = 1 -\lambda + 2\mu.
\]
$n^+$ is maximised when $\lambda = 0$, in which case 

\[
n^+ = n - \left(\frac{n + 2\mu}{2 + \mu}\right) \mbox{  and  } k = 1 + 2\mu.
\]
We can also use the well known identity for all SRGs that

\[
n = 1 + k + \frac{k(k - 1 - \lambda)}{\mu}.
\]
This implies that

\[
n^+ = n - \left(\frac{8(n - 1)}{8 + n}\right),
\]
which suggests the following conjecture.

\begin{conjecture}\label{conjecture2}
For any graph 
\[
n^+ \le  n - \left\lfloor\frac{8(n - 1)}{8 + n}\right\rfloor.
\]
\end{conjecture}

This bound is exact for connected graphs such as the Petersen graph and SRG(16, 5, 0 2), and for disconnected graphs such as $2C_5$, without use of the floor function. We have tested this bound against all named graphs with up to 40 vertices in the Wolfram Mathematica database and found no counterexample. Conjecture~\ref{conjecture2} outperforms Corollary~\ref{npmcorollary} for some graphs, such as the Petersen graph.

For large $n$, this bound has $n^+ \approx n$. Such graphs do exist. For example the Taylor family of graphs are $SRG(q^3, \frac{1}{2}(q - 1)(q^2 +1), \frac{1}{4}(q - 1)^3 - 1, \frac{1}{4}(q - 1)(q^2 + 1))$, where $q$ is an odd prime. Nikiforov \cite{nikiforov15} discusses the inertia of these graphs and notes that for sufficiently large $n$, almost all of the eigenvalues are positive, since

\[
n^+ = 1 + (q - 1)(q^2 + 1) = q^3 + q - q^2 = n + n^{\frac{1}{3}} - n^{\frac{2}{3}} \approx n.
\]

\section{Conclusions}

It is worth noting that equation (\ref{eq:pawel}) can be used to prove (generalisations) of the bounds due to Hoffman, Nikiforov and Kolotilina and the bound due to Lima \emph{et al} (see \cite{wocjan13} and \cite{elphick15}) as well as this inertial bound, but we do not know how to prove the Ando and Lin bound using (\ref{eq:pawel}).

In 1976 van Nuffelen conjectured that $\chi(G) \le \mathrm{rank}(A)$, but a counter-example was found on 64 vertices by Alon and Seymour in 1989, which has rank 29 and chromatic number of 32. The spectrum of the counter-example is $56^1 , 4^7 , 0^{35} , -4^{21}$, so this graph provides an example of the Hoffman bound significantly outperforming Theorem 1.

The bound in this paper is the first spectral bound for the chromatic number of which we are aware, which uses only the numbers of positive and negative eigenvalues of a graph.

\end{document}